\documentclass[final]{gCOV2e}

\theoremstyle{plain}% Theorem-like structures
\newtheorem{theorem}{Theorem}[section]
\newtheorem{corollary}[theorem]{Corollary}

\newtheorem{proposition}[theorem]{Proposition}

\theoremstyle{definition}

\newtheorem{example}[theorem]{Example}

\theoremstyle{remark}
\newtheorem{remark}[theorem]{Remark}

\def\M{\mathcal{M}}
\def\Level{\mathcal{L}}

\def\R#1{\mathbb{R}^{#1}}

\def\half#1#2{\begin{matrix}\frac{#1}{#2}\end{matrix}}
\def\F{\mathbb{A}}

\DeclareMathOperator{\trace}{\mathrm{tr}}
\DeclareMathOperator{\im}{\mathrm{Im}}
\DeclareMathOperator{\diver}{\mathrm{div}}
\DeclareMathOperator{\re}{\mathrm{Re}}

\begin{document}

\jvol{00} \jnum{00} \jyear{2018} \jmonth{November}

%\articletype{GUIDE}

\title{\textit{New construction techniques for  minimal surfaces}}

\author{
\name{ Jens Hoppe \textsuperscript{a}$^\flat$\thanks{$^\flat$ Email: jhoppe@ihes.fr} and Vladimir G. Tkachev\textsuperscript{b}$^{\ast}$\thanks{$^\ast$Corresponding author. Email: vladimir.tkatjev@liu.se}}
\affil{\textsuperscript{a} IHES, 35 Route des Chartres, F 91440 Bures Sur Yvette, France\\
\textsuperscript{b}
Department of Mathematics, Link\"oping University, 581 83 Link\"oping, Sweden}
\received{ }
}

\maketitle

\begin{abstract}
\textbf{Abstract.} It is pointed out that despite the non-linearity of the underlying equations, there do exist rather general methods that allow to generate new minimal surfaces from known ones.

\end{abstract}


\begin{classcode}Primary 53C42, 49Q05; Secondary 53A35\end{classcode}

\section{Introduction}

%\begin{figure}[h]
%\includegraphics[width=80mm]{Weiersstras03.png}
%\caption{A surface (\ref{eq:snsn}) for $k=4/5$ and its fundamental domain} \label{fig:snsn}
%\end{figure}

%\subsection{}
Whereas minimal surfaces in $\R{3}$ have been studied for more than 250 years, astonishingly little is known about higher dimensional minimal submanifolds \cite{Meeks}. Explicit examples are scarce, and until recently no general techniques were known to solve the underlying nonlinear PDEs.  The purpose of this note is to point out solution-generating techniques and in particular to note that there exist certain subclasses of solutions to the nonlinear (minimal surface) equations for which linear superposition principles exist, respectively elementary solution-generating operations, - somewhat similar (though not identical) to the celebrated B\"acklund-transformations\footnote{B\"acklund transformations (to which generically Bianchi-permutability theorems apply) for ordinary minimal surfaces actually do exist; they were studied by Bianchi and Eisenhart \cite{Bianchi}, \cite{Eisenhart2}, \cite{Eisenhart60} (partially referring to 'Thibault' transformations) and also more recently \cite{Corro}, \cite{Bak}, \cite{Bergfeldt} (under the name Ribaucour transformations; see also \cite{RogersS}) but their generalizations to higher dimensional minimal surfaces are unclear} that exist for some 'integrable' PDEs.

Our observations apply to both
\begin{itemize} \item[(G)] nonparametric (graph) hypersurfaces $$
\M(z)=\{(x,z(x))\in \R{N+1}: x\in \R{N}\}
$$
where $\Omega$ is an open subset and
\begin{equation}\label{eq1}
\diver \frac{\nabla z}{\sqrt{1+|\nabla z|^2}}
=
\frac{1}{(1+|\nabla z|^2)^{3/2}}\left((1+|\nabla z|^2)\Delta z-\Delta_\infty z\right)=0,
\end{equation}
and
\item[(L)] the level set description,
$$
\Level_C(u)=\{x\in \R{N}: u(x)=C\},
$$
where $C\in\R{}$ is an arbitrary constant and $u(x)$ is a solution of
\begin{equation}\label{eq2}
|\nabla u|^2\Delta u-\Delta_\infty u\equiv \Delta_1 u(x)=0.
\end{equation}
\end{itemize}
$\Delta_\infty$ is the $\infty$-Laplacian, defined by
\begin{equation}\label{eq3}
\Delta_\infty f=\frac12 \nabla f\cdot \nabla |\nabla f|^2,
\end{equation}
while for arbitrary finite $p$
$$
\Delta_p f=|\nabla f|^2\Delta f+(p-2) \Delta_\infty f.
$$
Note that (L) follows from the well-known relation for the mean curvature of the level set $\Level_0(u)$
\begin{equation}\label{Hdef}
H(x)=\diver \frac{\nabla u(x)}{|\nabla u(x)|}=|\nabla u(x)|^{-3}\Delta_1 u(x),
\end{equation}
see for example  \cite{Hsiang67}. In particular,

\begin{itemize}
\item[(L0)] if $\Delta_1 u(x)=\lambda(x)u(x)$ for some continuous (non-singular on $\Level_0(u)$ )
   function $\lambda$ then the level set $\Level_0(u)$ is a minimal hypersurface.
\end{itemize}

Note  that all the results below are local in  nature, unless stated explicitly otherwise.  We write, for instance, $x\in \R{N}$ to indicate that $x$ belongs to an open domain of $\R{N}$ rather than  whole $\R{N}$.

A function $f$ is said to be \textit{perfectly harmonic} if
\begin{equation}\label{eq4}
\Delta f=\Delta_\infty f=0 .\tag{PH}%\eqno{(PH)}.
\end{equation}

A function $f$ is perfectly harmonic if  $f$ is $p$-harmonic for at least two distinct $p$. It is easy to see that perfectly harmonic functions satisfy a remarkable superposition property (cf. \cite{Hop}):

\begin{proposition}[Superposition Principle]\label{pro:super}
If both $f(x)$ and $g(y)$ are perfectly harmonic in $\R{N}$ and $\R{M}$ respectively, so is any linear combination $h(x,y)=\alpha f(x)+\beta g(y)$ in $\R{N+M}$ for any $\alpha,\beta\in \R{}$.
\end{proposition}

The superposition principle allows one to combine perfectly harmonic functions to get diverse minimal hypersurfaces using \eqref{eq1} or \eqref{eq2}. Below we consider some applications of this method and discuss diverse examples of perfectly harmonic functions.

\section{Perfectly harmonic functions in lower dimensions}
In $\R{1}$, the only perfectly harmonic functions are the affine functions
$$
f_1=a+b_1x_1.
$$
Furthermore,  an old result of Aronsson (Theorem~2 in \cite{Aronsson68}) says that in the two-dimensional case the only perfectly harmonic functions in (a domain of) $\R{2}$ are affine functions $$
f_2=a+b_1x_1+b_2x_2
$$
and the (scaled and shifted) polar angle
\begin{equation}\label{eq5}
g_2=a+b \arctan \frac{x_2}{x_1}.
\end{equation}
This yields a complete description of perfectly harmonic functions in the first two dimensions $N\le 2$.

Applying the superposition principle to $f_1$ and $g_2$,  one obtains
$$
f_3(x_1,x_2,x_3)=x_3-\arctan \frac{x_2}{x_1},
$$
whose level sets are classical helicoids in $\R{3}$. Note that $f_3$ is a perfectly harmonic function in $\R{3}$.

Superposition of $g_2$ with an arbitrary perfectly harmonic function yields the following observation.

\begin{corollary}
Let $f(x)$ be a perfectly harmonic function in $\R{N}$. Then
$$
x_{N+2}=x_{N+1}\tan f(x)
$$
defines a minimal hypersurface in $\R{N+2}$.
\end{corollary}

\begin{proof}
The function $g=\arctan (x_{N+2}/x_{N+1})$
is a perfectly harmonic function of two independent variables $x_{N+1}$ and $x_{N+2}$, hence by Proposition~\ref{pro:super}
$$
h(x,x_{N+1},x_{N+2}):=f(x)-\arctan\frac{x_{N+2}}{x_{N+1}}
$$
is a perfectly harmonic function in $\R{N+2}$, hence by virtue of (L) the zero level set $h=0$ is a minimal hypersurface, as desired.
\end{proof}

It follows immediately from (\ref{eq1}) and (\ref{eq2}) that to any perfectly harmonic function, one can associate several different minimal hypersurfaces:

%\tikzset{
%   basic/.style  = {draw, text width=5.9cm,
%   %drop shadow,
%   font=\sffamily, rectangle},
%   root/.style   = {basic, rounded corners=2pt, thin, align=center,
%   %fill=green!5
%                    },
%   level 2/.style = {basic, rounded corners=6pt, thin,align=center, %fill=green!5
%   }
%}
%\begin{figure}[h]
%\begin{tikzpicture}[
%   level 1/.style={sibling distance=50mm},
%   edge from parent/.style={->,draw},
%   >=latex]
%
%% root of the the initial tree, level 1
%\node[root] {$\vphantom{\int^{1}_{2}}$ perfectly harmonic function(s) $f$}
%% The first level, as children of the initial tree
%   child {node[level 2,text width=3cm] (c1) {$\Level_C(f)$ in $\R{N}$}}
%   child {node[level 2,text width=3cm] (c2) {$\mathscr{M}(f)$ in $\R{N+1}$}}
%   child {node[level 2,text width=5cm] (c3) {$\M(x_{N+1}\tan f(x))$ in $\R{N+2}$}}
%   ;
%\end{tikzpicture}
%\end{figure}

\begin{enumerate}
 \item[(i)] an (L)-representation $\Level_C(f)$ in $\R{N}$;
\item[(ii)] a (G)-representation $\M(f)$ in $\R{N+1}$;
 \item[(iii)] a (G)-representation $\M(x_{N+1}\tan f(x))$ in $\R{N+2}$.
\end{enumerate}

\section{Eigenfunctions of $\Delta_1$}\label{sec:eigen}

Let us  write
\begin{equation}\label{eigenf1}
u(x)\equiv v(x)\mod w(x)
\end{equation}
if $u(x)-v(x)=\mu(x)w(x)$ for some continuous $\mu$.

A $C^2$-function $u(x)$ is said to be an \textit{eigenfunction} of $\Delta_1$ if the zero-level set of $u(x)$ has a  nonempty regular part and
\begin{equation}\label{eigenf}
\Delta_1 u(x)\equiv 0\mod u(x).
\end{equation}
The (continuous) function $\mu(x):=\Delta_1 u(x)/u(x)$ is said to be the \textit{weight} of the eigenfunction.

%Some remarks are in order.

\begin{remark}
Our definition of eigenfunction of $\Delta_1$ should not be confused with the standard  definition for the $p$-Laplacian, $1<p<\infty$ which comes from a variational problem. On the other hand, it is well known that the variational definition does not work for the exceptional case $p=1$.
\end{remark}

Of course, the condition \eqref{eigenf} is nontrivial only along the zero set of $u$. Then the regularity  condition in the above definition implies by virtue of (L) and \eqref{eigenf} the (regular part of) zero-level set of an eigenfunction is a minimal hypersurface.

Hsiang \cite{Hsiang67} was probably the first to consider  (algebraic) eigenfunctions of $\Delta_1$. In particular, he asked to classify all homogeneous degree 3 polynomials with weight function $|x|^2$, radial eigencubics in terminology \cite{Tk10c}. It is known that the latter problem is intimately connected to Jordan algebras \cite{NTVbook}.

Below we consider some elementary properties of eigenfunctions of $\Delta_1$.
Let $\phi(t)$ is a $C^2$-function such that $\phi(0)=0$. Then it follows immediately from the first identity in \eqref{Hdef} that if $u$ is an eigenfunction of $\Delta_1$, so is  the composition $\phi(u(x))$, and
$$
\Delta_1 \phi(u(x))=\psi(x)\phi(u(x)),
$$
where the  weight function $\psi(x)=\mu(x)\frac{u(x)\phi'^3(u(x))}{\phi(u(x))}$ is obviously continuous.

A less trivial observation is that multiplication by \textit{any} smooth function preserves the property being eigenfunction. More precisely, we have

\begin{proposition}
\label{pro:delta}
Let $v(x)$ be a $C^2$-function, $v(x)\ne0$. Then $u(x)$ is an eigenfunction of $\Delta_1$ if and only if $v(x)u(x)$ is.
\end{proposition}

\begin{proof}
We claim that
\begin{equation}\label{eqgf3}
\Delta_1 uv\equiv v^3 \Delta_1 u\mod u.
\end{equation}
Indeed, we have $|\nabla (uv)|^2\equiv v^2|\nabla u|^2\mod u,$
\begin{equation*}
\begin{split}
\Delta uv&\equiv v\Delta u+2\nabla u\cdot \nabla v \mod u,
\end{split}
\end{equation*}
and applying  an obvious identity $\nabla (uw)\equiv w\nabla u\mod u$ we obtain
\begin{equation*}
\begin{split}
\Delta_\infty uv&\equiv \half12 v\nabla u\cdot \nabla(v^2|\nabla u|^2 +2uv \nabla u\cdot \nabla v+u^2|\nabla v|^2)\\
&\equiv \half12 v\nabla u\cdot \bigl(\nabla(v^2|\nabla u|^2) +2v (\nabla u\cdot \nabla v)\nabla u\bigr)\\
&\equiv \half12 v\nabla u\cdot \bigl(2v|\nabla u|^2\nabla v+v^2\nabla |\nabla u|^2 +2v (\nabla u\cdot \nabla v)\nabla u\bigr)\\
&\equiv v^3\Delta_\infty u+2v^2|\nabla u|^2\nabla u\cdot \nabla v \mod u.
\end{split}
\end{equation*}
Combining the obtained formulas yields (\ref{eqgf3}).

Next, let $\Delta_1 u(x)=\mu(x)u(x)$ for some continuous $\mu$. We have from \eqref{eqgf3} that there exists some continuous $\nu(x)$ such that
$$
\Delta_1 uv=v^3 \Delta_1 u+\nu u=\biggl(\frac{\nu}{v}+v^2\mu\biggr) uv
$$
which implies the "only if" conclusion. The "if" statement follows by replacing $v\to 1/v$.
\end{proof}

The geometric meaning of  the made observation is clear: the zero-level set of both the composed function  $\phi(u(x))$ and the product $v(x)u(x)$ coincides with that of $u(x)$.

\section{Perfectly harmonic functions from orthogonal twin-harmonics}
As we have already seen, perfectly harmonic functions can be thought as building blocks to construct minimal hypersurfaces. This motivates the problem to classify all perfectly harmonic functions. In this section and below we discuss some particular results in this direction.

Our first step is to generalize  (\ref{eq5}). Two functions $u(x)$ and $v(x)$ in $\R{N}$ are said to be \textit{orthogonal twin-harmonics} if
\begin{align}
\frac{\Delta u}{u}&=\frac{\Delta v}{v}, \label{eq9} \\
  \frac{\Delta_\infty u}{u}&=\frac{\Delta_\infty v}{v},\label{eq10}\\
|\nabla u|^2=|\nabla v|^2, &\quad \nabla u\cdot \nabla v=0.\label{eq11}
\end{align}
Notice that \eqref{eq9} and \eqref{eq11} together imply that $u$ and $v$ are eigenfunctions of $\Delta_1$ with the same weight function.

It is easy to see that rotations and dilatations preserve the property being orthogonal twin-harmonics. More precisely, if $\textbf{w}(x):=(u(x),v(x))^t$ are orthogonal twin-harmonics in $\R{N}$ then so are the pairs $T \textbf{w}(x)$, where $T$ is an element of the linear conformal group (i.e. $T^tT=cI$ for some real $c\ne0$). In general we have

\begin{proposition}
\label{pro:tan}
If $(u(x),v(x))$ are orthogonal twin-harmonics in $\R{N}$ then
\begin{equation}\label{eq8}
f(z)=\arctan \frac{v(x)}{u(x)}
\end{equation}
is a perfectly harmonic function in $\R{N}$, wherever the arctan function is well defined.
\end{proposition}

\begin{proof}
   Let us denote $w:=|\nabla u|^2=|\nabla v|^2$. Then $\nabla f=\frac{1}{u^2+v^2}(u\nabla v-v\nabla u)$, and,
using (\ref{eq9}) resp. (\ref{eq11}),
   $|\nabla f|^2=\frac{w}{u^2+v^2}$, as well as
$$
\Delta f=\frac{1}{u^2+v^2}(u\Delta v-v\Delta u) -\frac{2}{(u^2+v^2)^2}(u\nabla u+v\nabla v)\cdot(u\nabla v-v\nabla u)=0.
$$
Furthermore,
$$
\nabla |\nabla f|^2=\frac{\nabla w}{u^2+v^2}-\frac{2w(u\nabla u+v\nabla v)}{(u^2+v^2)^2}.
$$
It follows from (\ref{eq10}) that
$$
0=v\Delta_\infty u-u\Delta_\infty v=\frac12( v\nabla u -u\nabla v)\cdot \nabla w
$$
hence (using again (\ref{eq11}))
$$
2\Delta_\infty f= \frac{1}{u^2+v^2}(u\nabla v-v\nabla u)\cdot (\frac{\nabla w}{u^2+v^2}-\frac{2w(u\nabla u+v\nabla v)}{(u^2+v^2)^2})=0,
$$
as desired.
\end{proof}

An important example is the generalized helicoid constructed in \cite{CH} , which was noticed in \cite{LeeLee} to be of the form (\ref{eq8}), with $u$ and $v$ being specific twin harmonics, homogenous of degree $2$.

\begin{proposition}
\label{pro:tan1}
Let $(u(x),v(x))$ be orthogonal twin-harmonics  in $\R{N}$, $h(x)=u(x)+iv(x)$ and $k\in \R{}$. Then $(\re h(x)^k, \im h(x)^k)$ are orthogonal twin-harmonics in $\R{N}$. In particular, $(u^2(x)-v^2(x), 2u(x)v(x))$ are orthogonal twin-harmonics  in $\R{N}$.
\end{proposition}

\begin{proof}
Let us rewrite (\ref{eq10}) as
\begin{equation}\label{eq10n}
\nabla w \cdot \nabla u=\rho u, \qquad \nabla w \cdot \nabla v=\rho v,
\end{equation}
where $w=|\nabla u|^2=|\nabla v|^2$. Let $f(z)$ be a holomorphic function and let  $U(x)=\re f\circ h(x)$ and $V(x)= \im f\circ h(x)$. Then by the holomorphy of $h$
\begin{equation}\label{re1}
\begin{split}
\nabla U&=\re f'(h) \nabla u-\im f'(h)\nabla v\\
\nabla V&=\im f'(h) \nabla u+\re f'(h)\nabla v,
\end{split}
\end{equation}
hence
\begin{equation}\label{nablau}
|\nabla U|^2=|\nabla V|^2=|f'(h)|^2w, \quad \nabla U\cdot \nabla V=0,
\end{equation}
implying (\ref{eq11}). Also using (\ref{eq11}),
$$
\Delta U=\re f''(h) ( |\nabla u|^2-|\nabla v|^2)=0,
$$
and similarly $\Delta V=0$ which yields (\ref{eq9}). Finally, since $\ln |f'(h)|=\re \ln f'(h)$, one has
\begin{equation*}
\begin{split}
\nabla |f'(h)|^2&=2|f'(h)|^2\nabla \ln |f'(h)|=
2|f'|^2\left(\re \frac{f''}{f'} \nabla u-\im \frac{f''}{f'}\nabla v\right)\\
&=2\left (\re (f''\overline{f'})\,\nabla u-\im (f''\overline{f'})\,\nabla v\right )
\end{split}
\end{equation*}
hence by (\ref{eq11})
$$
\nabla |f'(h)|^2\cdot \nabla U=2w\re (f''(h)\overline{f'^2(h)})
$$
Similarly, using (\ref{eq10n}) and (\ref{re1}),
$$
\nabla w \cdot \nabla U=\rho (u\re f'(h) -v\im f'(h))=\rho\re hf'(h),
$$
Therefore (\ref{re1}) and (\ref{nablau}) yield
\begin{equation*}
\begin{split}
\nabla  |\nabla U|^2\cdot \nabla U&=|f'(h)|^2 \nabla w\cdot \nabla U+w\nabla |f'(h)|^2\cdot \nabla U\\
&=\rho|f'(h)|^2\re hf'(h)+2w^2\re (f''(h)\overline{f'^2(h)}).
\end{split}
\end{equation*}
Similarly one finds
$$
\nabla |\nabla V|^2 \cdot \nabla V=\rho|f'(h)|^2\im hf'(h)-2w^2\im (f''(h)\overline{f'^2(h)}),
$$
thus readily implying
\begin{equation*}
\begin{split}
V\Delta_\infty U-U\Delta_\infty V&=-\frac12\rho|f'(h)|^2\im hf'(h)\overline{f(h)}+w^2\im (f''(h)f(h)\overline{f'^2(h)})\\
&=-\frac12\rho|f'(h)f(h)|^2\im \frac{hf'(h)}{f(h)}+w^2|f'(h)|^4\im \frac{f''(h)f(h)}{f'^2(h)}
\end{split}
\end{equation*}
hence the latter expression vanishes if $f(h)=h^k$, $k\in \R{}$ as desired.
\end{proof}

\section{Orthogonal twin-harmonics in even dimensions}\label{sec:orth}

Let $h(z)$ be a  holomorphic function of $z\in \mathbb{C}^m$. We write $h\in \mathcal{T}^m$ if
\begin{equation}\label{eq13}
\sum_{i,j=1}^m\overline{h_{z_iz_j} }h_{z_i}h_{z_j}=\mu(z) {h(z)}
\end{equation}
for some real-valued function $\mu(z)$ that is regular on $h(z) = 0$. An essence of the introduced class follows from the following observation.

\begin{proposition}\label{pro:51}
If $h\in \mathcal{T}^m$ then $(\re h(z),\im h(z))$ are orthogonal twin-harmonics in $\R{2m}\cong \mathbb{C}^m$.
\end{proposition}

\begin{proof}
Write $h(z)=u(z)+iv(z)$. Then the Cauchy--Riemann equations yield
\begin{equation}\label{ident1}
\begin{split}
u_{x_k} &=\,\,\,\,v_{y_k} =\re  h_{z_k} \\
u_{y_k} &=-v_{x_k} =-\im  h_{z_k}
\end{split}
\end{equation}
and
\begin{equation}\label{ident2}
\begin{split}
u _{x_px_q} &=v _{x_py_q} =-u _{y_py_q} =\re h _{z_pz_q}\\
v _{x_px_q} &=-u _{x_py_q} =-v _{y_py_q} =\im h _{z_pz_q}.
\end{split}
\end{equation}
Using (\ref{ident1}) and (\ref{ident2}) we obtain that $\nabla u\cdot \nabla v=0$,
$$
|\partial_zh|^2:=\sum_{p=1}^m |h_{z_p}|^2=|\nabla u|^2=|\nabla v|^2,
$$
and also
\begin{equation}\label{reh}
\Delta_\infty u+i\Delta_\infty v=\sum_{p,q=1}^{m} \overline{h_{z_pz_q} }h_{z_p}h_{z_q}=\mu (u+iv),
\end{equation}
Hence both $u$ and $v$ satisfy (\ref{eq9})--(\ref{eq11}), the desired conclusion follows.
\end{proof}

Note that by (\ref{ident2}), $u$ and $v$  are harmonic functions, hence one has in fact from (\ref{reh}) and (\ref{eq2}) that
$$
\Delta_1 u=-\mu(z)u, \quad \Delta_1 v=-\mu(z)v,
$$
hence $u$ and $v$ are eigenfunctions of $\Delta_1$ with the same weight. Then (L0) implies

\begin{corollary}\label{cor:im}
If $h\in \mathcal{T}^m$ then the zero-level sets $\Level(\re  h)$ and $\Level(\im h)$ are minimal hypersurfaces in $\R{2m}\cong \mathbb{C}^m$.
\end{corollary}

We consider some further examples.

\begin{example}\label{ex:1}
Any linear function obviously satisfies (\ref{eq13}). A less trivial example is that the quadratic form
$$
h(z_1,\ldots,z_m)=z_1^2+\ldots+z_m^2
$$
satisfies (\ref{eq13}) with $\mu= 8$. The corresponding minimal hypersurface is the Clifford cone
$$
\re h(z)=x_1^2+\ldots +x_m^2-y_1^2-\ldots-y_m^2=0.
$$
The conjugate minimal hypersurface is given by
$$
\im h(z)=2(x_1y_1+\ldots +x_my_m)=0,
$$
and these two Clifford cones are congruent, i.e. coincides under an orthogonal transformation of $\R{2m}$. In Section~\ref{sec:Hsiang} below we consider an example of a cubic form satisfying \eqref{eq13}. In that case, the corresponding  conjugate hypersurfaces are no longer congruent.
\end{example}

\begin{example}\label{pro:det}This  example provides an  irreducible homogeneous polynomial solution to (\ref{eq13}) of arbitrary high degree.
Let $Z\in \mathbb{C}^{m^2}$ denote the matrix with entries $z_{ij}$, $1\le i,j\le m$. Then $\det Z$ is irreducible over $\mathbb{C}$ \cite{Lang} and $\det Z\in \mathcal{T}^{m^2}$. Indeed, setting $f(z)=\det (z_{ab})$ we have  by the Jacobi formula
$$
\frac{\partial f}{\partial z_{ij}}=f z^{ji},
$$
where $z^{ij}$ denotes the $(i,j)$-entry of the inverse matrix $Z^{-1}$. We have
\begin{align*}
\frac{\partial^2 f}{\partial z_{ij}\partial z_{kl}}&=f z^{lk}z^{ji}-f\sum_{i',j'=1}^m z^{jj'}\frac{\partial z_{j'i'}}{\partial z_{kl}}z^{i'i}=f(z^{ji}z^{lk}-z^{jk}z^{li}).
\end{align*}
therefore
\begin{align*}
\sum_{i,j,k,l=1}^m\overline{f_{z_{ij}z_{kl}} }f_{z_{ij}}f_{z_{kl}}&=
f|f|^2\sum_{i,j,k,l=1}^m  z^{ji}z^{lk}(\overline{z^{ji}z^{lk}-z^{jk}z^{li}})\\
&=\frac12 f|f|^2\sum_{i,j,k,l=1}^m  (z^{ji}z^{lk}-z^{jk}z^{li})(\overline{z^{ji}z^{lk}-z^{jk}z^{li}})=f\mu
\end{align*}
which proves (\ref{eq13}) with a real-valued function $$
\mu=-\half12  |f|^2\sum_{i,j,k,l=1}^m|z^{ji}z^{lk}-z^{jk}z^{li}|^2\equiv -\half12\trace (D^2f \cdot \overline{D^2f})
$$
and the desired property follows.
\end{example}

\begin{remark}
A similar proof applies to the  Pfaffian of a generic skew-symmetric matrix considered in \cite{HopLT}.
%; see also \cite{CH}, \cite{LeeLee}.
\end{remark}

Let us consider properties of the class $ \mathcal{T}^m$ in more detail. Combining Proposition~\ref{pro:tan} with  Proposition~\ref{pro:51} yields

\begin{corollary}
\label{cor:rhol}
If $h(z)\in \mathcal{T}^m$ then $\arg h(z)$ is a perfectly harmonic function in $\R{2m}\cong \mathbb{C}^m$.
\end{corollary}

The next proposition shows that $\mathcal{T}^{m}$ has nice multiplicative properties.

\begin{proposition}\label{pro:4}
Let $h(z)\in \mathcal{T}^{m}$ and $g(w)\in \mathcal{T}^{n}$. Then
\begin{itemize}
\item[(i)]
$c\,h(z)^r\in \mathcal{T}^{m}$ for any $c\in \mathbb{C}$ and $r\in \R{}$;

\item[(ii)]
$h(z)g(w)\in \mathcal{T}^{m+n}$.

\item[(iii)]
$h(z)/g(w)\in \mathcal{T}^{m+n}$.
\end{itemize}
\end{proposition}

\begin{proof}
Setting $H(z):=h(z)^r$ one has
$$
\sum_{i,j=1}^{m} \overline{H_{z_iz_j} }H_{z_i}H_{z_j}=
r^3((r-1) |h|^{2r-4}|D_zh|^4+\mu |h|^{2r-2})h^r=\mu_1 H,
$$
where $D_z(h)=(\frac{\partial h}{\partial z_1},\ldots, \frac{\partial h}{\partial z_m})$. Then  $\mu_1=r^3\sum_{i,j=1}^{m} ((r-1) |h|^{2r-4}|D_zh|^4+\mu |h|^{2r-2})$ is obviously a real-valued function, thus implying $h^r\in \mathcal{T}^{m}$. Similarly one justifies $ch\in \mathcal{T}^{m}$ which yields (i).
Next, we have
$$
\sum_{i,j=1}^{m} \overline{h_{z_iz_j} }h_{z_i}h_{z_j}=\mu h, \qquad
\sum_{\alpha,\beta=1}^{n} \overline{g_{w_{\alpha}w_{\beta}} }g_{w_\alpha}g_{w_\beta}=\nu g,
$$
where $\mu=\mu(z)$ and $\nu=\nu(w)$ are real-valued functions. Therefore, setting $H(z,w):=h(z)g(w)$ one obtains
\begin{equation*}
\begin{split}
\sum_{i,j=1}^{m+n} \overline{H_{z_iz_j} }H_{z_i}H_{z_j}=&
|g|^2g\sum_{i,j=1}^{m} \overline{h_{z_iz_j} }h_{z_i}h_{z_j}+
|h|^2h\sum_{\alpha,\beta=1}^{n} \overline{g_{w_{\alpha}w_{\beta}} }g_{w_\alpha}g_{w_\beta}\\
&+2|D_zh|^2|D_w g|^2\,hg=(\mu |g|^2+\nu|h|^2+2|D_zh|^2|D_w g|^2)H,
\end{split}
\end{equation*}
implying (ii). Finally, setting $r=-1$ and $c=1$ in (i) implies that $1/h(z)\in \mathcal{T}^n$, thus together with (ii) implies (iii).
\end{proof}

We finish this section by demonstrating some particular solutions to \eqref{eq13}. We start with the complete characterization of class $\mathcal{T}^{1}$. It would be interesting to obtain a similar characterization for $\mathcal{T}^{m}$, $m\ge2$.

\begin{proposition}
\label{pro:m}
Any element of $\mathcal{T}^1$ is either a binomial $h(z)=(az+b)^p$ or the exponential $h(z)=e^{pz}$, where  $a,b\in \mathbb{C}$ and $p\in \R{}$.
\end{proposition}

\begin{proof}
Indeed, let $\Omega$ be the domain of holomoprhy of $h(z)$. Then (\ref{eq13}) yields
$$
\frac{|h'' (z)|^2}{\mu(z)}=\frac{h'' (z)h(z)}{h'^2(z)},
$$
where the right-hand side is a meromorphic function in $\Omega$, while the left-hand side is real valued in $\Omega$. Thus, the both sides are  constant in $\Omega$, say equal to $c\in \R{}$. This yields $ch'^2(z)=h''(z)h(z)$, or $h'(z)=ch(z)^b$ for some real $b$. This  yields the required conclusions.
\end{proof}

Composing the above solutions with Proposition~\ref{pro:4} yields some more  examples.

\begin{example}
We demonstrate how the above facts apply to construct  minimal hypersurfaces in \textit{odd-dimensional} ambient spaces.
Let $h(z)\in \mathcal{T}^m$. Then
\begin{equation}\label{arg1}
x_{2m+1}=\arg h(z)
\end{equation}
is a minimal hypersurface in $\R{2m+1}$.
Indeed, the function $g(z_1,\ldots,z_m,z_{m+1}):=ie^{z_{m+1}}h(z_1,\ldots,z_m)$ is $\R{}$-holomorphic by Proposition~\ref{pro:4} and Proposition~\ref{pro:m}. Then
$$
\re g=-e^{\re z_{m+1}}(\re h\sin\im  z_{m+1}+\im h\cos\im  z_{m+1})
$$
yields that $\re g=0$ is equivalently defined by
$$
  \im z_{m+1}=-\arctan \frac{\im h}{\re h}=-\arg h+C
$$
for some real constant $C$. It is easily seen that the latter equation is equivalent to  (\ref{arg1}) up to an orthogonal transformation (a reflection) of $\R{2m+1}$.
\end{example}

\begin{example}
Setting $h(z_1)=z_1=x_1+i x_2$,  (\ref{arg1}) becomes $x_3=\arctan x_2/x_1$, i.e. the classical helicoid. More generally, one has the following minimal hypersurface:
$$
x_{2m+1}=\arg (z_1^{k_1}...z_m^{k_m}), \qquad k_i\in \mathbb{Z}.
$$
\end{example}

Combining Example~\ref{ex:1} and Proposition~\ref{pro:4} yields.
\begin{corollary}
\label{cor:Lawson}
Let natural numbers $p_i$, $1\le i\le m$, be subject to the GCD condition $(p_1,\ldots,p_m)=1$ and let $c\in \mathbb{C}^\times$. Then the hypersurface $$
\re (cz_1^{p_1}\cdots z_m^{p_m})=0,
$$
  is minimal (in general singularly) immersed cone in $\R{2n}\cong \mathbb{C}^m$.
\end{corollary}

\begin{example}\label{ex:3}
For $c=1$, Corollary~\ref{cor:Lawson} yields exactly the observation made earlier by H.B.~Lawson~\cite[p.~352]{Lawson}. For instance, when $m=2$ one obtains the well-known infinite family of immersed algebraic minimal Lawson's hypercones $\re (z_1^pz_2^q)=0$, $(p,q)=1$, in $\R{4}$. The intersection of such a cone with the unit sphere $S^3$ is an immersed minimal surface of Euler characteristic zero of $S^3$ \cite{Lawson}.
\end{example}

\begin{example}\label{ex:4}
Using $c=\sqrt{-1}$ in  Corollary~\ref{cor:Lawson} yields minimal hypersurfaces in $\R{2m}$ of the following kind:
$$
\sum_{i=1}^m p_i\arctan \frac{y_k}{x_k}=0
$$
which obviously is an algebraic minimal cone in $\R{2m}$.
\end{example}

\section{Perfectly harmonic functions via Hsiang eigencubics}\label{sec:Hsiang}

A cubic form $u(x)$ on $\R{N}$ is called a \textit{Hsiang eigencubic} (or \textit{radial eigencubic} according to \cite{Tk10c}, \cite{NTVbook}) if
\begin{equation}\label{Hs1}
\Delta_1 f=\lambda |x|^2 f,\qquad x\in \R{N}
\end{equation}
for some $\lambda \in \R{}$. Here $|x|^2=x_1^2+\ldots+x_N^2$. According to the definition of Section~\ref{sec:eigen}, $f$ is an eigenfunction of $\Delta_1$. In this section we construct  Hsiang eigencubics $f$ which are also in $\mathcal{T}^N$. Then it follows by Proposition~\ref{pro:51} that each such eigencubic $f$ gives rise to orthogonal twin-harmonics.

The simplest example of such Hsiang eigencubic  is given explicitly by
\begin{equation}\label{F0}
f_0(x)=x_1x_2x_3,
\end{equation}
the verification that $f_0\in \mathcal{T}^3$ is straightforward. The further examples treated here require us to evolve some  ideas developed in \cite{Tk14} and \cite{NTVbook}. First, in order to
derive the differential relations on $f$, we need to view the
ambient spaces $\R{N}$ as the traceless subspaces of certain Jordan algebras. This viewpoint, developed below, unifies the triplet  of Hsiang eigencubics, and brings out the rich geometric and algebraic
structure embodied there (we just indicate some elementary observations).

\begin{remark}\label{rem:Hs}
Though, there are infinitely many non-congruent cubic homogeneous solutions of \eqref{Hs1},  only the four examples of Hsiang eigencubics constructed below belong to $\mathcal{T}^N$. More precisely, one can prove that the latter property holds if and only if the Peirce dimension $n_1=0$, see \cite{NTVbook}. We do not give any details of the corresponding proofs, as they require a more deep treatment of Jordan algebras and the Hsiang cubic cones theory.
\end{remark}

In order to construct examples, we begin by describing the Jordan algebra viewpoint mentioned above,
and then using this viewpoint to derive the corresponding differential relations. The standard references here are \cite{McCrbook}, \cite{FKbook}.  Let $W=\mathfrak{h}_r(\mathbb{A}_d)$ denote the algebra over the reals on the vector space of all Hermitian matrices of size $r$ over the Hurwitz division algebra $\F_d$, $d\in \{1,2,4,8\}$ (the reals $\F_1=\R{}$, the complexes $\F_2=\mathbb{C}$, the Hamilton quaternions $\F_4=\mathbb{H}$ and the Graves--Cayley octonions $\F_8=\mathbb{O}$) with the multiplication
$$
x\bullet y=\half12(xy+yx),
$$
where $xy$ is the standard matrix multiplication. Then it is classically known that if $r\le 3$ and $d\in \{1,2,4,8\}$, or if $r\ge 4$ and $d\in \{1,2,4\}$ then $W$ is a \textit{simple} Jordan algebra with the unit matrix $e$ being the algebra unit. It is easy to see that
$$
N+1:=\dim_{\R{}} W=r+\frac{r(r-1)}{2}d.
$$
The algebra $W$ is commutative and for $r\ge 2$  is nonassociative. Still,  the algebra $W$ is \textit{power associative}, i.e. the subalgebra generated by any element is always associative. In particular, the multiplicative powers  $x^{\bullet n}$ do not depend on associations:  $x^{\bullet 2}=x\bullet x$, $x^{\bullet 3}=x\bullet x^{\bullet2}$, $x^{\bullet 2}\bullet x^{\bullet 2}=x^{\bullet 3}\bullet x$, etc. This readily yields that for any  $x\in W$ the Hamilton--Cayley identity holds:
\begin{equation}\label{HamCayley}
x^{\bullet N}-\sigma_1(x)x^{\bullet (N-1)}+\sigma_2(x)x^{\bullet (N-2)}+\ldots+(-1)^{N} \sigma_{N}(x)e=0,
\end{equation}
where the coefficients $\sigma_k(x)$ are real-valued homogeneous degree $k$ functions of $x$.
The first coefficient $\sigma_1(x)$ is called the linear trace form and it is \textit{associative}, i.e.
\begin{equation}\label{asso}
\sigma_1((x\bullet  y)\bullet z)=\sigma_1(x\bullet (y\bullet  z)), \qquad \forall x,y,z\in W,
\end{equation}
i.e. $\sigma_1(x\bullet  y\bullet z)$ is well-defined without
parentheses.
The higher degree forms $\sigma_i(x)$ are recovered from $\sigma_1$  by virtue of the Newton identities:
\begin{equation}\label{Newton}
\begin{split}
\sigma_2(x)&=\half12(\sigma_1(x)^2-\sigma_1(x^{\bullet 2})),\\
\sigma_3(x)&=\half16(\sigma^3_1(x)-3\sigma_1(x^{\bullet 2})\sigma_1(x)+2\sigma_1(x^{\bullet 3})),\\
\sigma_4(x)&=\half1{24}(\sigma_1(x)^4- 6\sigma_1(x^{\bullet 2})\sigma_1(x)^2+3\sigma^2_1(x^{\bullet 2})+8\sigma_1(x)\sigma_1(x^{\bullet 3})- 6\sigma_1(x^{\bullet 4})),\\
\end{split}
\end{equation}

In the remained part of the section we always assume that $r=4$ and consider a subspace of consisting of trace-free elements:
$$
V=\{x\in \mathfrak{h}_4(\mathbb{A}_d):\, \sigma_1(x)=0\}, \qquad N=\dim V=3+6d,\qquad d\in \{1,2,4\}.
$$
Then  we find from \eqref{Newton}
\begin{equation}\label{Newton1}
\begin{split}
\sigma_2(x)&=-\half12\sigma_1(x^{\bullet 2}),\quad
\sigma_3(x)=\half13\sigma_1(x^{\bullet 3}),\quad
\sigma_4(x)=\half1{8}(\sigma^2_1(x^{\bullet 2})- 2\sigma_1(x^{\bullet 4})),
\end{split}
\end{equation}
hence  \eqref{HamCayley} yields
\begin{equation}\label{HamCayley1}
%x^{\bullet 4}=-\sigma_2(x)x^{\bullet 2}+\sigma_3(x)x-\sigma_{4}(x)e,
x^{\bullet 4}=\half12\sigma_1(x^{\bullet 2})x^{\bullet 2}+\half13\sigma_1(x^3)x-\half1{8}(\sigma^2_1(x^{\bullet 2})- 2\sigma_1(x^{\bullet 4}))e,
\end{equation}
Multiplying \eqref{HamCayley1} by $x$ followed by application of the trace form $\sigma_1$, one obtains
\begin{equation}\label{HamCayley3}
\sigma_1(x^{\bullet 5})=\half56\sigma_1(x^{\bullet 2})\sigma_1(x^{\bullet 3}).\end{equation}

Next, consider the real-valued bilinear form
$$
b(x,y)=\sigma_1(x\bullet y), \qquad x,y\in W.
$$
Since $\bullet$ is commutative, $b(x,y)=b(y,x)$, and $b(x,x)=\sigma_1(x\bullet x)>0$ unless $x=0$ because $x$ is a Hermitian matrix. Thus, $b$ is a positive definite inner product on $W$.  In particular, $W$ is a (simple) Euclidean Jordan algebra, cf. \cite{FKbook} Furthermore, from \eqref{asso}
\begin{equation}\label{asso1}
b(x\bullet y, z)=b(x, y \bullet z), \qquad \forall x,y,z\in W,
\end{equation}
Also, since $\sigma_1(x)=\sigma_1(x\bullet e)=b(x,e),$ we see that $
V$ is the orthogonal complement to $e$ in $W$: $V=e^\bot$. Note that
$
b(e,e)=\sigma_1(e)=\trace e=4,
$
hence $\half12 e$ is the unit vector in $W$. Below we make frequently use the following orthogonal projection formula:
$$
x^V=x-\half14b(x,e)e=x-\half14\sigma_1(x)e: W\to V,
$$
where $x\to x^V:W\to V$ is the orthogonal projection.

Let $K=\R{}$ or $K=\mathbb{C}$. Let $\{e_i\}_{1\le i\le N}$ be an orthonormal (w.r.t. $b$) basis of $V$.  Since $b(e_i\bullet e_j,e_k)\in \R{}$, the function
$$
f(z_1,\ldots, z_{N}):=\half16\sigma_1(z_e^{\bullet 3})=\half16b(z_e^{\bullet2}, z_e), \qquad z_e:=\sum\nolimits_{i=1}^{N}z_ie_i\in V,
$$
is a cubic form of $(z_1,\ldots, z_{N})\in K^{N}$ with real coefficients, i.e. $f\in \R{}[z_1,\ldots,z_{N}]$. In particular, if $K=\mathbb{C}$ then $f$ is a holomorphic function of $z$.

\begin{proposition}\label{pro:FH}
$f$ is a Hsiang eigencubic in $\R{N}$. Furthermore, $f\in \mathcal{T}^{N}$.
\end{proposition}

\begin{proof}
We have for the directional derivative $\partial_{z_i}z_e=e_i$, therefore
\begin{equation}\label{direct}
\begin{split}
f_{z_j}&=\half16(b(e_j\bullet z_e,z_e)+b(z_e\bullet e_j,z_e)+b(z_e\bullet z_e,e_j))=\half12b(z_e\bullet z_e,e_j)\\
&=\half12b(z_e^{\bullet2},e_j),
\end{split}
\end{equation}
and similarly
\begin{equation}\label{direct1}
f_{z_jz_k}=\half12(b(e_k\bullet z_e,e_j)+b(z_e\bullet e_k,e_j))= b(z_e, e_j\bullet  e_k).
\end{equation}
Also, if $K=\mathbb{C}$ then since $f_{z_jz_k}$ is a linear form in $z_e$ with real coefficients, we also have
\begin{equation}\label{direct2}
\overline{f_{z_jz_k}}=b(\bar z_e, e_j\bullet  e_k), \qquad \text{where }\quad \bar z_e:=\sum\nolimits_{i=1}^{N}\bar z_ie_i.
\end{equation}
For a general field $K$, we have for the Laplacian
$$
\Delta f=\sum_{j=1}^Nf_{z_jz_j}=\sum_{j=1}^Nb(z_e, e_j\bullet  e_j)
$$
The extended system $\{e_i\}_{0\le i\le N}$ with $e_0=\half12 e$ is an orthonormal basis of $W$. Since the algebra $W$ is a \textit{simple} Euclidean Jordan algebra, one has $\sum_{i=0}^{N}e_i^{\bullet 2}=\half{N+1}{4}\,e$, see \cite [Exercise~6, p.~59]{FKbook}. It follows that $\sum_{i=1}^{N}e_i^{\bullet 2}=\half{N}{4}e$, therefore $$\Delta f=\half{N}{4}b(z_e,e)=0$$ by virtue of $z_e\in V$. This proves that $f$ is harmonic.

To optimize the further calculations, we find the following sum:
\begin{equation}\label{summ}
S:=\sum_{j,k=1}^Nb(w_e, e_j\bullet  e_k)b(z_e^{\bullet 2}, e_j)
b(z_e^{\bullet 2}, e_k), \qquad w_e=\sum\nolimits_{i=1}^{N} w_ie_i,
\end{equation}
where $(w_1,\ldots, w_N)\in K^N$. Note that
$$
\sum_{k=1}^{N}b(x,e_k)b(y,e_k) =b(x,y)-b(x,e_0)b(y,e_0)=
b(x,y)-\frac{b(x,e)b(y,e)}{4}
$$
for any $x,y\in W$.
This yields by virtue of $b(z_e,e)=0$ that
\begin{align*}
\sum_{k=1}^{N}b(w_e, e_j\bullet e_k) b(z_e^{\bullet2},e_k) &=\sum_{k=1}^{N}b(w_e\bullet e_j, e_k) b(z_e^{\bullet2},e_k)\\
&=b(w_e\bullet e_j,z_e^{\bullet2}) -\half14b(w_e\bullet e_j,e)b(z_e^{\bullet2},e)\\
&=b(w_e\bullet z_e^{\bullet2},e_j)-\half14b(w_e, e_j)b(z_e,z_e),
\end{align*}
therefore, arguing similarly we obtain
\begin{align*}
S(w_e,z_e)&=\sum_{j=1}^{N}\biggl(b(w_e\bullet z_e^{\bullet2},e_j)-\half14b(w_e, e_j)b(z_e,z_e)\biggr)
b(z_e^{\bullet2},e_j)\\
&=b(w_e\bullet z_e^{\bullet2}, \, z_e^{\bullet2})
-\half14b(w_e\bullet z_e^{\bullet2},e)b(z_e^{\bullet2},e)
-\half14b(w_e, z_e^{\bullet2})b(z_e,z_e)\\
&=b(w_e, \, z_e^{\bullet4})
-\half12b(w_e, z_e^{\bullet2})b(z_e,z_e)\\
&=b(w_e, \, z_e^{\bullet4}-\half12 \sigma_1(z_e^{\bullet2})z_e^{\bullet2})\\
&=b(w_e, \,\half13 \sigma_1(z_e^{\bullet3})z_e)\qquad \text{by \eqref{HamCayley1}}\\
&=2f(z)b(w_e, \,z_e).
%&=b(w_e, \, z_e^{\bullet2}\bullet z_e^{\bullet2})\\
%&=b(w_e, \, z_e^{\bullet4})=b(w_e, \,
%\half12\sigma_1(x^{\bullet 2})x^{\bullet 2}+\half13\sigma_1(x^3)x)
\end{align*}
Thus,
\begin{equation}\label{ssum}
S(w_e,z_e)=\sum_{j,k=1}^Nb(w_e, e_j\bullet  e_k)b(z_e^{\bullet 2}, e_j)
b(z_e^{\bullet 2}, e_k)=2f(z)b(w_e, \,z_e).
\end{equation}
Taking into account the harmonicity of $f$ and  setting $w_e=z_e$ in \eqref{ssum}  yields
$$
\Delta_1 f=-\half14S(z_e,z_e)=-\half12b(z_e, \,z_e)f(z)=-\half12(z_1^2+\ldots+z_N^2)f(z),
$$
which proves that $f$ is a Hsiang eigencubic (with $\lambda=-\half12$ in \eqref{Hs1}). Next, assuming that $K=\mathbb{C}$ and setting $w_e=\bar z_e$ in \eqref{ssum},
we get
\begin{equation}\label{Teq}
\sum_{j,k=1}^m\overline{f_{z_iz_j} }f_{z_i}f_{z_j}=\half14S(\bar z_e,z_e)=\half12f(z)b(w_e, \,z_e)=\half12(z_1\bar z_1+\ldots+z_N\bar z_N)f(z),
\end{equation}
hence $f\in \mathcal{T}^N$, as desired. The proposition follows.
\end{proof}

In summary, Proposition~\ref{pro:FH} yields three (noncongruent) Hsiang eigencubics $f_d$ such that $f_d\in \mathcal{T}^N$, each in dimensions $N=3+6d$, where $d=1,2,4$. It is convenient to think of the cubic form $f_0$ given by \eqref{F0} as the member corresponding to $d=0$. In fact, these $f_d$, $d\in \{0,1,2,4\}$ are exactly the only possible  Hsiang eigencubics having the Peirce dimension $n_1=0$, see Remark~\ref{rem:Hs} above.

In the rest of this section, we exhibit some explicit representations. Let $\F_d$ denote an associative  real division algebra of dimension $d\in \{1,2,4\}$, and $\F_0=0$. A generic element of  $W=\mathfrak{h}_4(\mathbb{A}_d)$ is the Hermitian matrix
$$
x:=\left(
     \begin{array}{cccc}
       w_1 & z_1 & z_3 & z_5 \\
       \bar z_1 & w_2 & z_6 & z_4 \\
       \bar z_3 & \bar z_6 & w_3 & z_2 \\
       \bar z_5 & \bar z_4 & \bar z_2 & w_4 \\
     \end{array}
   \right), \qquad z_i\in \F_d,\,w_i\in \R{}.
$$
Then the linear trace form is determined by
$$
\sigma_1(x)=\trace x=w_1+w_2+w_3+w_4,
$$
hence the inner product quadratic form is defined by
$$
b(x,x)=\sigma_1(x^{\bullet 2})=\sigma_1(x^2)=\sum_{i=1}^4w_i^2+2\sum_{j=1}^6|z_j|^2.
$$
The trace free subspace $V$ of $W$ consists of the matrices $x_v$ with $w_i$ given  by
\begin{equation}\label{www}
\begin{split}
w_1&=\half1{\sqrt{2}}(v_1+v_2+v_3),\\
w_2&=\half1{\sqrt{2}}(v_1-v_2-v_3),\\
w_3&=\half1{\sqrt{2}}(-v_1+v_2-v_3),\\
w_4&=\half1{\sqrt{2}}(-v_1-v_2+v_3),
\end{split}\end{equation}
where $v=(v_1,v_2,v_3)\in\R{3}$. In this notation
\begin{equation}\label{fdd}
f_d=\half16\sigma_1(x_v^{\bullet 3})=\half16 \trace x_v^3
\end{equation}

In the trivial case $d=0$, all off-diagonal terms vanish: $z_k=0$, hence
$$
f_0=\half16 \sum_{i=1}^4w_i^3=\sqrt{2}v_1v_2v_3,
$$
cf. \eqref{F0}. The first nontrivial example is obtained for $d=1$:
\begin{align*}
f_1&=\sqrt {2}v_1v_2v_3+\half{1}{\sqrt{2}}\left( (z_1^2-z_2^2)v_1+(z_3^2-z_4^2)v_2+(z_5^2-z_6^2)v_3 \right) \\ &+z_{{2}}z_{{4}}z_6+z_{{2}}z_{{3}}z_{{5}} +z_{{1}}z_{{3}}z_6+z_{{1}}z_{{4}}z_{{5}}, \qquad\qquad\qquad\qquad z_j,w_k\in\R{}.
\end{align*}
It is straightforward to verify that in the new orthonormal coordinates
$$
z'_{2i-1}=\pm\half1{\sqrt{2}}(z_{2i-1}+z_{2i}), \quad
z'_{2i}=\pm\half1{\sqrt{2}}(z_{2i-1}-z_{2i}),\qquad i=1,2,3,
$$
where the signs $\pm$ are chosen appropriately,
$f_1$ becomes the usual determinant
$$
f=\left|
    \begin{array}{ccc}
      v_1 & z_6' & z_3' \\
      z_5' & v_2 & z_2' \\
      z_4' & z_1' & v_3 \\
    \end{array}
  \right|,
$$
thus bringing us back to the determinant varieties discussed  in  Example~\ref{pro:det} for $m=3$.

Remarkably,  the case $d=2$ ($N=15$) can also be interpreted in related terms, namely as the Pfaffian $P_3$ in notation of \cite{HopLT}. We also remark that the eigencubics in dimensions $N=9$ and $N=15$ were discovered by another  method by  Wu-Yi Hsiang, see Examples 1 and 2 in \cite{Hsiang67}. The nature of the 27-dimensional example is more subtle. We mention that this example was constructed  explicitly   using octonions by Liu Tongyan in \cite{Liu}.

Finally, remark that it follows from Proposition~\ref{pro:51}

\begin{corollary}
\label{pro:last}
If $f_d$ is a Hsiang cubic defined by \eqref{fdd} over $K=\mathbb{C}$ then  $(\re f_d(z),\im f_d(z))$ are orthogonal twin-harmonics in $\R{6+12d}$, $d=0,1,2,4$.
\end{corollary}

We briefly illustrate the latter property, we consider the  simplest particular case $d=0$, when $f_0(x)=x_1x_2x_3$ (dropping off the non-essential constant factor). Then
$$
f_0(x_1+ix_4,x_2+ix_5,x_3+ix_6)=u(x)+iv(x), \quad x=(x_1,\ldots,x_6)\in \R{6},
$$
and the corresponding twin-harmonics are
\begin{align*}
u(x)&=x_1x_2x_3-x_1x_5x_6-x_4x_2x_6-x_3x_4x_5,\\
v(x)&=x_1x_2x_6+x_1x_3x_5+x_2x_3x_4-x_4x_5x_6
\end{align*}
are the desired orthogonal twin harmonics in $\R{6}$. It is straightforward to verify that $u$ and $v$ satisfy (\ref{Hs1}) with $\lambda=-2$.

\section{Acknowledgement}

We thank H.Lee and V.V.Sergienko for collaboration at initial stages of our investigations, and  the referee for constructive suggestions.

%\bibliographystyle{plain}%{amsalpha}
%%%\bibliography{}
%\bibliography{main_references}

\def\cprime{$'$}

\end{document}